\documentclass{amsart}

\usepackage{amsfonts}
\usepackage{amsmath}
\usepackage{amssymb}
\usepackage{amsthm}
\usepackage{color}

\newtheorem{Thm}{Theorem}
\newtheorem{Lem}[Thm]{Lemma}
\newtheorem{Prop}[Thm]{Proposition}
\newtheorem{Cor}[Thm]{Corollary}

\theoremstyle{definition}
\newtheorem{Def}[Thm]{Definition}

\theoremstyle{remark}
\newtheorem{Rem}[Thm]{Remark}

\numberwithin{equation}{section}

\def\CC{{\mathbb{C}}}

\def\RR{{\mathbb{R}}}
\def\ZZ{{\mathbb{Z}}}
\def\NN{{\mathbb{N}}}
\def\PP{{\mathbb{P}}}

\def\PP{{\mathbb{P}}}

\title[Bounding the number of remarkable values via Jouanolou's theorem]{Bounding the number of remarkable values via Jouanolou's theorem}
\date{\today}

\author[G.~Ch\`eze]{Guillaume Ch\`eze}
\address{Institut de Math\'ematiques de Toulouse\\
Universit\'e Paul Sabatier Toulouse 3 \\
MIP B\^at 1R3\\
31 062 TOULOUSE cedex 9, France}
\email{guillaume.cheze@math.univ-toulouse.fr}

\begin{document}
	
\begin{abstract}
In this article we bound the number of remarkable values of a polynomial vector field. The proof is short and based on  Jouanolou's theorem about rational first integrals of planar polynomial derivations. Our bound is given in term of the size of a Newton polygon associated to the vector field.  We prove that this bound is almost reached.
\end{abstract}	
	
\maketitle

\section*{Introduction}
In this paper we study polynomial differential systems  in $\CC^2$ of this form:
$$\dfrac{dX}{dt}=A(X,Y), \, \, \, \dfrac{dY}{dt}=B(X,Y),$$
where $A, B \in \CC[X,Y]$ are coprime and $\deg (A), \deg(B) \leq k$.\\
 We associate  to these polynomial differential systems the polynomial derivations\\ $D=A(X,Y)\partial_{X}+B(X,Y)\partial_Y$. \\
 
 The computation and the study of  rational first integrals of such  polynomial differential systems is an old and classical problem. We recall that a rational first integral is a rational function $f/g\in  \CC(X,Y)$ such that the curves $\lambda f-\mu g=0$, where $(\lambda:\mu) \in \PP^1(\CC)$, give  orbits of the differential system. Thus it is  a function $f/g \in \CC(X,Y)$ such that $D(f/g)=0$.\\
 
 When we study rational first integrals, we can always consider a rational first integral $f/g$ with a minimum degree. We recall that the degree of $f/g$ is equal to the maximum of $\deg(f)$ and $\deg(g)$. \\
  Rational first integrals with minimum degree are \emph{indecomposable} rational functions.  We recall that a rational function is \emph{decomposable} when it can be written $u(h)$ where $u\in \CC(T)$, $h \in \CC(X,Y)$ and $\deg(u) \geq 2$, otherwise $f/g$ is said to be \emph{indecomposable}.  Sometimes indecomposable rational functions are called minimal rational functions.\\
    In 1891, Poincar\'e has shown that if we consider an indecomposable rational first integral  $f/g$, then the number of $(\lambda:\mu) \in \PP^1(\CC)$ such that  $\lambda f -\mu g$ is reducible in $\CC[X,Y]$ or $\deg(\lambda f - \mu g) <\deg(f/g)$, is finite. Poincar\'e called these kind of values  $(\lambda:\mu)$ : \emph{remarkable values}. Poincar\'e was interested by the intersection of  different level sets of a given rational first integral. He has also shown how to use remarkable values in order to study the inverse integrating factor, see \cite{Poincare}. Recently, new results have been given in this direction, see \cite{chavarriga_giac_gine_llibre,FerragutLlibre,FerragutLlibreMahdi,CollFerragutLlibre,Ferragutnew}. In \cite{GarciaLlibreRio,GarciaGiacRio}, remarkable values  are used to study systems with polynomial first integrals. In \cite{AlvarezFerragut} they are used to study  degenerate singular points. 
 They also plays a role in an algorithm computing the decomposition of multivariate rational functions, see \cite{Chezedecomp}, and in another algorithm computing rational first integrals with bounded degree, see \cite{BCCW}.\\
 
Poincar\'e has shown, when all the singular points of the polynomial vector field are distinct, that the number of remarkable values is bounded by the number of saddle points plus 2, see \cite{Poincare}. Then a direct application of Bezout's theorem shows  that the number of remarkable values is bounded by $k^2+2$.\\

 Since Poincar\'e, a lot of authors, see below, have given bounds about the number of remarkable values and have studied the problem of reducibility in a pencil of algebraic curves. To the best of our knowledge all these bounds are given in term of the degree $d$ of the rational function $f/g$. Even if there exists a relation between $k$ and $d$, see below, this means that a geometric point of view about the pencil $\lambda f - \mu g$ is generally used. In this note, we consider the pencil $\lambda f - \mu g$ as level sets of a rational first integrals. With this point of view we can use results about rational first integrals and get easily a new bound about the number of remarkable values. Furthermore, our strategy allows us to give a bound on the  \emph{total order of reducibility}. This number is defined in the following way:\\ 
We denote by $f^h$ and $g^h$ the homogeneous polynomial of degree $\deg(f/g)$ in $\CC[X,Y,Z]$ associated to $f$ and $g$ and we set:
$$ \sigma(f,g)= \{(\lambda:\mu) \in \PP^1(\CC) \mid \lambda f^h - \mu g^h \mbox{ is reducible in } \CC[X,Y,Z]\}.$$
This set is called the \emph{spectrum} of $f/g$.\\ 
We have introduced the homogeneous polynomial $f^h$ and $g^h$ in order to have a uniform definition for the remarkable values. The situation where $\deg( \lambda f - \mu g)$ is smaller than $\deg(f/g)$ corresponds to the situation where $Z$ divides $\lambda f^h - \mu g^h$. Thus, by definition of the spectrum,  \emph{the spectrum is the set of all remarkable values.}\\

 If $(\lambda:\mu) \in \sigma(f,g)$  then  we have 
   $$\lambda f^h -\mu g^h= \prod_{i=1}^{n(\lambda:\mu)}\big(f_{(\lambda:\mu),i}^h\big)^{e_{(\lambda:\mu),i}},$$
   where $f_{(\lambda:\mu),i}^h$ is the homogeneous polynomial in $\CC[X,Y,Z]$ associated to the irreducible factor $f_{(\lambda:\mu),i}$ of $\lambda f - \mu g$.
The \emph{total order of reducibility} is 
$$\rho(f,g)=\sum_{(\lambda:\mu) \in \sigma(f,g)}\big( n(\lambda:\mu) -1\big).$$

For example, if $f=Y$ and $g=X^2$ then $\sigma(f,g)= \{(0:1), (1:0)\}$. Indeed, $X^2$ is reducible and $1=\deg(f)<\deg(f/g)=2$. Moreover, $f^h=Y.Z$, then $n(1:0)=2$ and $g^h=x^2$ then $n(0:1)=1$. Thus $\rho(f,g)=(2-1) +(1-1)=1$ and with this example we have: $\rho(f,g) <|\sigma(f,g)|.$ \\
We get this inequality because $X^2$ is reducible and has not distinct irreducible factors. Thus if we want to bound the number of remarkable values with $\rho(f,g)$ we have to also add  the number of $(\lambda:\mu) \in \PP^1(\CC)$ such that $\lambda f^f-\mu g^h$ is a pure power. Then we introduce:

$$ \gamma(f,g)= \{(\lambda:\mu) \in \PP^1(\CC) \mid \lambda f^h - \mu g^h =P^e \mbox{ where  }  e>1  \mbox{ and  } P \in \CC[X,Y,Z]\}.$$

Furthemore the number of elements in $\gamma(f,g)$ is smaller or equal to 3, see \cite{Joua_Pfaff,BuseCheze,Abhyankar} for a proof and \cite[Remark 8]{Abhyankar} for an example where the bound is reached.

\begin{Rem}\label{rem1}
We have:
$$ |\sigma(f,g)| \leq \rho(f,g)+|\gamma(f,g)|\leq \rho(f,g)+3.$$
\end{Rem}
Thus we are going to bound $\rho(f,g)$ and it will give a bound on the number of remarkable values.\\

As our theorem uses Newton polygons, we recall that the Newton polygon of a Laurent polynomial $f(X,Y)=\sum_{\alpha,\beta} c_{\alpha,\beta}X^{\alpha}Y^{\beta}$, where  $(\alpha,\beta) \in \ZZ^2$, is the convex hull in $\RR^2$ of the exponents $(\alpha,\beta)$ of all  nonzero terms of $f$.\\
 
 \begin{Thm}\label{Thm}
Let $D=A(X,Y)\partial_{X}+B(X,Y)\partial_Y$ be a derivation, such that\\ $\deg(A), \deg(B) \leq k$, and let $f/g \in \CC(X,Y)$ be   an indecomposable rational function which is a first integral of $D$.\\  Consider a generic point $(x,y)$ in $\CC^2$ and  the Newton polygon $N_D$ associated to the Laurent polynomial $x\dfrac{A(X,Y)}{X}+y\dfrac{B(X,Y)}{Y}$, and denote by $\mathcal{B}$  the number of integer points in $N_D\cap \NN^2$, then:
$$\rho(f,g)< \mathcal{B}+2.$$
\end{Thm}

Theorem \ref{Thm} is well suited for "sparse" derivations i.e.  when some coefficients of $A$ and $B$ are equal to zero.  For example when we consider polynomials $A$ and $B$ of this form: $c_{e,e}X^eY^e+c_{e-1,e}X^{e-1}Y^e+c_{e,e-1}X^{e}Y^{e-1} +c_{0,0}$, where $e$ is an integer strictly bigger than $1$, then $\mathcal{B}=3e+2$, and $k=\deg(A)=\deg(B)=2e$.  Thus for such examples Theorem \ref{Thm} gives a linear bound in $e$. Then \emph{the bound is linear} in the degree of the derivation. \\The Newton polygon associated to these polynomials $A$ and $B$ is given in Figure \ref{fig1}. The Newton polygon $N_D$ associated to this example is given in Figure \ref{fig2}. In these figures, small points correspond to coefficients equal to zero.\\

\begin{figure}[!h]
\begin{center}
\setlength{\unitlength}{.4cm}
\begin{picture}(10,10)
\put(0.5,1){$0$}
\put(1,2){\vector(1,0){7}}
\put(1,1.5){\vector(0,1){7}}
\put(-0.5,8){$Y$}
\put(8.5,1.5){$X$}
\put(1,2){\circle*{.2}}
\put(2,3){\circle*{.2}}
\put(3,4){\circle*{.2}}
\put(4,4){\circle*{.2}}
\put(3,5){\circle*{.2}}
\put(3,2){\circle*{.1}}
\put(2,2){\circle*{.1}}
\put(4,2){\circle*{.1}}
\put(1,3){\circle*{.1}}
\put(1,4){\circle*{.1}}
\put(1,5){\circle*{.1}}
\put(1,6){\circle*{.1}}
\put(1,7){\circle*{.1}}
\put(2,4){\circle*{.1}}
\put(2,5){\circle*{.1}}
\put(2,6){\circle*{.1}}
\put(3,2){\circle*{.1}}
\put(3,3){\circle*{.1}}
\put(4,3){\circle*{.1}}
\put(5,3){\circle*{.1}}
\put(5,2){\circle*{.1}}
\put(6,2){\circle*{.1}}
\put(7,2){\circle*{.1}}
\put(6,3){\circle*{.1}}
\put(5,4){\circle*{.1}}
\put(4,5){\circle*{.2}}
\put(3,6){\circle*{.1}}
\put(2,7){\circle*{.1}}
\put(1,8){\circle*{.1}}
\thicklines\put(1,2){\line(3,2){3}}
\thicklines\put(1,2){\line(2,3){2}}
\thicklines\put(3,5){\line(1,0){1}}
\thicklines\put(4,5){\line(0,-1){1}}
\end{picture}
\caption{Newton polygon \mbox{$\mathcal{N}\big(c_{3,3}X^3Y^3+c_{2,3}X^{2}Y^3+c_{3,2}X^{3}Y^{2} +c_{0,0}\big)$}}
\label{fig1}
\end{center}
\end{figure}

\vspace{1cm}

\begin{figure}[!h]
\begin{center}
\setlength{\unitlength}{.4cm}
\begin{picture}(10,10)
\put(0.5,1){$0$}
\put(1,2){\vector(1,0){7}}
\put(1,1.5){\vector(0,1){7}}
\put(-0.5,8){$Y$}
\put(8.5,1.5){$X$}
\put(1,2){\circle*{.2}}
\put(2,3){\circle*{.2}}
\put(3,4){\circle*{.2}}
\put(4,4){\circle*{.2}}
\put(3,5){\circle*{.2}}
\put(3,2){\circle*{.1}}
\put(2,2){\circle*{.2}}
\put(4,2){\circle*{.1}}
\put(1,3){\circle*{.2}}
\put(1,4){\circle*{.1}}
\put(1,5){\circle*{.1}}
\put(1,6){\circle*{.1}}
\put(1,7){\circle*{.1}}
\put(2,4){\circle*{.2}}
\put(2,5){\circle*{.2}}
\put(2,6){\circle*{.1}}
\put(3,2){\circle*{.1}}
\put(3,3){\circle*{.2}}
\put(4,3){\circle*{.2}}
\put(5,3){\circle*{.1}}
\put(5,2){\circle*{.1}}
\put(6,2){\circle*{.1}}
\put(7,2){\circle*{.1}}
\put(6,3){\circle*{.1}}
\put(5,4){\circle*{.1}}
\put(4,5){\circle*{.1}}
\put(3,6){\circle*{.1}}
\put(2,7){\circle*{.1}}
\put(1,8){\circle*{.1}}
\thicklines\put(3,5){\line(1,-1){1}}
\thicklines\put(2,2){\line(2,1){2}}
\thicklines\put(1,3){\line(1,2){1}}
\thicklines\put(2,5){\line(1,0){1}}
\thicklines\put(4,4){\line(0,-1){1}}
\thicklines\put(1,2){\line(1,0){1}}
\thicklines\put(1,2){\line(0,1){1}}

\end{picture}
\caption{The associated Newton polygon $N_D$.}
\label{fig2}

\end{center}
\end{figure}
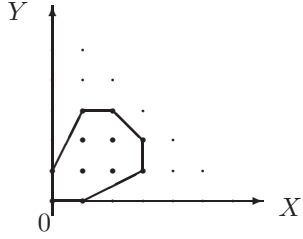

If we consider dense polynomials $A$, $B$ with degree $k$, that is to say each coefficient of $A$ and $B$ is nonzero, then $\mathcal{B}=k(k+1)/2$. We get then a \emph{quadratic bound} (in term of the degree $k$ of the derivation) on $\rho(f,g)$.\\

As the number of remarkable values is $|\sigma(f,g)|$ then by Remark \ref{rem1}  we get:
\begin{Cor}\label{Cormain}
With the previous notations, we have:\\
The number of remarkable values is smaller than $\mathcal{B}+2+|\gamma(f,g)|$, more precisely
$$|\sigma(f,g)| <\mathcal{B}+2+|\gamma(f,g)| \leq \mathcal{B}+5.$$
\end{Cor}

This bound is almost reached. Indeed, consider:
$$D=(X^k-1)\partial_X-(kX^{k-1}Y+1)\partial_Y$$
An indecomposable polynomial first integral of $D$ is $f(X,Y)=Y(X^k-1)+X$. The remarkable values associated to $D$ are $(1:\omega)$ where $\omega^k=1$ and $(0:1)$. Indeed, here the denominator $g$ is equal to $1$, thus $\deg(0f-1g)=\deg(1)<\deg(f/g)$.  We deduce that in this situation $k+1 \leq |\sigma(f,1)|$.
Furthermore, for this derivation we have $\mathcal{B}=k$,  then by Remark \ref{rem1}, this gives 
$$\mathcal{B}+1\leq |\sigma(f,1)| \leq \rho(f,1)+|\gamma(f,1)|.$$
Here $|\gamma(f,1)|=1$, because $0f^h-1g^h=Z^{k+1}$. The point $(0:1)$ is the only possible point in $\gamma(f,1)$. Indeed, if $(\lambda:\mu) \neq (0:1) $ and $(\lambda:\mu)\in \gamma(f,1)$ then by definition  we have  $\lambda f-\mu=P^e$. As this implies  $f$ decomposable, we deduce that $|\gamma(f,1)|=1$. Then Corollary \ref{Cormain}  gives
 $$\mathcal{B}+1\leq |\sigma(f,1)|<\mathcal{B}+3.$$
 Thus the bound is almost reached.\\
 We can also remarked that  in this example we have $\rho(f,1)= \mathcal{B}$, and Theorem \ref{Thm} give $\rho(f,1)\leq \mathcal{B}+1$.\\
 
At last, we  remark that we can always consider a rational function $f/g$ as a rational first integral. Indeed, $f/g$ is a rational first integral of the jacobian derivative
$$D_{f/g}=A_{f/g}\partial_X-B_{f/g}\partial_Y,$$
where $$A_{f/g}=\dfrac{\partial_Y(f)g-f\partial_Y(g)}{\gcd\big(\partial_Y(f)g-f\partial_Y(g),\partial_X(f)g-f\partial_X(g)\big)}$$
 and
  $$B_{f/g}=\dfrac{\partial_X(f)g-f\partial_X(g)}{\gcd\big(\partial_Y(f)g-f\partial_Y(g),\partial_X(f)g-f\partial_X(g)\big)}.$$

 Thus if we want to study the reducibility in the pencil $\lambda f - \mu g$ then we can always suppose that  $f/g$ is an indecomposable rational first integral of a derivation.\\
 \subsection*{Related results}
 The study of remarkable values corresponds to the study of the irreducibility in a pencil of algebraic plane curves. It is an old problem and it has been widely studied since Bertini, see \cite{Kleiman}. It seems that Bertini and Poincar\'e has proved independently the following result:
 \begin{Thm}\label{spectre-fini}
 The spectrum associated to an indecomposable rational function is finite.
 \end{Thm}

 To the best of our knowledge, the first author, after Poincar\'e, who has  given a bound on the spectrum  was Ruppert, in \cite{Ruppert}. He has shown that the number of remarkable values associated to an indecomposable rational function of degree $d$ is smaller than $d^2-1$. Ruppert's strategy is based on the computation of the first de Rham's cohomology group of the  complementary of a plane curve. This approach gives an effective method to compute the spectrum.\\
Lorenzini has also studied the spectrum of  a rational function in \cite{Lorenzini} and he has shown that $\rho(f,g) \leq d^2-1$.\\
In \cite{BuseCheze}, the authors have used Ruppert's approach and have shown that the bound $d^2-1$ still works if we take into account the multiplicities of the factors.\\ 
Theorem \ref{Thm} can be seen as a counterpart of this bound with multiplicities. Indeed, in \cite{Poincare}, Poincar\'e has introduced the remarkable factor
$$R(X,Y)=\prod_{(\lambda:\mu) \in \sigma(f,g)} \prod_{i=1}^{n(\lambda:\mu)}f_{(\lambda:\mu),i}^{e_{(\lambda:\mu),i}-1}$$
 and has also  given the following relation: 
$$\deg(f)+\deg(g)-1=k+\deg(R).$$
 For a proof of this equality, we can read \cite{FerragutLlibre}.\\
  Thus a bound in term of the degree of the rational function which takes into account the multiplicities of the irreducible factors is in the same vein than a bound in term of  the degree $k$ of the vector field. The contribution of this present paper is to give a short proof which gives a nearly optimal bound in the sparse case.\\ 
Stein, in \cite{Stein},  considers the polynomial case. He has shown that $\rho(f,1) \leq d$.  It seems that the word "spectrum" has been introduced by Stein. This expression is  also used by other authors. For example this expression is used in \cite{BuseCheze} because in this paper the set of remarkable values corresponds to the spectrum of a pencil of matrices.\\
The strategy used by Stein was the following: first construct a rational first integral with some factors $f_{(1:\mu),i}$ and then get a contradiction. The construction of the first integral was obtained with geometrical arguments specific to the polynomial case (i.e. $g=1$). This approach has been extended   to the rational case by Bodin in \cite{Bodin}. The bound obtained is  $\rho(f,g)\leq d^2+d$. In this note, we are going to use the same approach. However, the construction of the rational first integral will be a direct consequence of a theorem due to Jouanolou.  This allows to get a bound in term of  $k$,  the degree of the derivation, and not in term of $d$, the degree of the rational function.  Furthermore, our proof is direct and gives a nearly optimal result for sparse derivations.\\

Some authors have already used the Darboux theory of integrability in order to show that the spectrum is finite. Moulin-Ollagnier, in \cite{Ollagnierdecomp}, has  proved  the finiteness of the spectrum by studying the number of distinct cofactors of a derivation.  Moulin-Ollagnier calls the factor $f_{(\lambda:\mu),i}$ "small Darboux polynomials".\\
We can also mention the paper \cite{chavarriga_giac_gine_llibre}, where the authors show with a simple proof using Darboux theory of integrability that the spectrum is finite.\\
 Unfortunately, these approachs do not give  sharp bounds.\\

There exist many other papers about the spectrum, for more details read e.g. \cite{BodinDebesNajib2,BodinDebesNajib1,NajibActa,Najibnvar,BuseChezeNajib,Abhyankar,Kaliman,Vistoli}.

\subsection*{Structure of this paper}
 In Section  \ref{toolbox} we recall some results about invariant algebraic curves and Jouanolou's theorem. In Section \ref{sec:indecomp}  we recall some classical results used in the proof of Theorem \ref{Thm}. In Section \ref{sec:proof}, we prove Theorem \ref{Thm}.


\subsection*{Notations}
$\NN=\{0;1;2\ldots\}$ is the set of integer numbers.\\
$|S|$ is the number of elements in the set $S$.\\
In the following $k$ will denote the degree of the derivation and $d$ the degree of a rational first integral.\\
In this paper, when we consider a rational function $f/g$, we always suppose that $f$ and $g$ are coprime. Furthermore, we recall that $\deg(f/g)$ means $\max\{\deg(f),\deg(g)\}$.
\section{Invariant algebraic curves and Jouanolou's theorem}\label{toolbox}
 In 1878, G.~Darboux \cite{Darboux} has given a strategy to find first integrals. One of the tools developed by G.~Darboux is now called \emph{invariant algebraic curves} and it will be the main ingredient in our proof.\\
\begin{Def}
A polynomial $f$ is said to be an invariant algebraic curve associated to $D$,  if $D(f)=\Lambda.f$, where $\Lambda$ is a polynomial. The polynomial $\Lambda$ is called the cofactor of $f$. 
\end{Def}
There exist a lot of different names in the literature for invariant algebraic curves, for example we can find: Darboux polynomials, special integrals, eigenpolynomials, special polynomials, or second integrals. \\
A lot of properties of a polynomial differential system are related to invariant algebraic curves of the corresponding derivation $D$, see e.g. \cite{Goriely,Dumortier_Llibre_Artes}. \\

 G.~Darboux shows in \cite{Darboux} that \emph{if the derivation  $D$ has at least \mbox{$k(k+1)/2+1$} irreducible invariant algebraic curves then $D$ has a  first integral which can be expressed by means of these polynomials}. More precisely the first integral has the following form: $\prod_i f_i^{c_i}$ where $f_i$ are invariant algebraic curves and $c_i$ are complex numbers. This kind of integral is called nowadays a Darboux first integral. \\
If the all the $c_i$ belong to $\ZZ$ then we have a rational first integral. That is to say a first integral which belongs to  $\CC(X,Y)$. The relation between rational first integral and invariant algebraic curves is given in the following proposition:

\begin{Prop}\label{fact-inv}
 If $f/g$ is a rational first integral then for all $(\lambda:\mu) \in \PP^1(\CC)$, $\lambda f -\mu g$ is an invariant algebraic curve and this curve corresponds to the level set $f/g=\mu/\lambda$.\\
Furtermore,   the irreducible factors $f_{(\lambda:\mu),i}$ of $\lambda f -\mu g$ are also invariant algebraic curves.
\end{Prop}
\begin{proof}
The first part is a straightforward computation. The second part come from the fact that irreducible factors of an invariant algebraic curve are also invariant algebraic curves, see e.g. \cite[Proposition 8.4]{Dumortier_Llibre_Artes} or \cite[Proposition 2.5]{Goriely}.
\end{proof}

 As mentioned before we are going to use a theorem due to Jouanolou.
  This theorem  is the following, see \cite{Joua_Pfaff}:
 
 \begin{Thm}
 If a derivation has at least $k(k+1)/2+2$ irreducible invariant algebraic curves then  there exist integers $n_i \in \ZZ$ such that $\prod_i f_i^{n_i} \in \CC(X,Y)$ is a rational first integral.
 \end{Thm}

Several authors have given simplified proof of Jouanolou's theorem. M.~Singer proves this result in    \cite{Singer}.   A direct proof  of Jouanolou's result is also given in \cite{CerveauMattei} and in  \cite{Dumortier_Llibre_Artes}. \\
Darboux and Jouanolou theorems are improved in \cite{LZ1,LZ2}. The authors show that we get the same kind of result if we take into account the multiplicity of invariant algebraic curves. The multiplicity of an invariant algebraic curve is defined and studied in \cite{Christopher_Llibre_Pereira}. This notion of multiplicity does not  correspond to the multiplicity $e_{(\lambda:\mu),i}$ of the irreducible factors  of $\lambda f - \mu g$. \\

Darboux and Jouanolou's theorems  have been also studied in the sparse case, i.e. bounds are given in terms of a Newton polytope in \cite{ChezeSparse}. More precisely, the result is the following: 

\begin{Thm}\label{cheze-sparse}
If a derivation has at least $\mathcal{B}+2$ irreducible invariant algebraic curves, where $\mathcal{B}$ is defined as in Theorem \ref{Thm}, then  there exist integers $n_i \in \ZZ$ such that $\prod_i f_i^{n_i} \in \CC(X,Y)$ is a rational first integral.
\end{Thm}

The key point in the proof of this theorem is the study of the Newton polygon of the cofactors  for a given derivation, for other results about the structure of the cofactors see \cite{FerragutGasull}.


\section{Indecomposable rational functions and remarkable values}\label{sec:indecomp}
\begin{Def}
We say that a rational  function $f/g \in \CC(X,Y)$   is \emph{decomposable} if we have $f/g=u(h)$  where $u (T)\in \CC(T)$  with $\deg(u) \geq 2$ and $h(X,Y) \in \CC(X,Y)$, otherwise $f/g$  is said to be \emph{indecomposable}.\\
\end{Def}

Sometimes indecomposable rational functions are called minimal rational functions.
 
\begin{Rem}\label{remhomo}When $f/g=u(h)$ with $h$ indecomposable, we can suppose that \mbox{$h=h_1/h_2$} with $h_1$ and $h_2$ irreducible in $\CC[X,Y]$ and $\deg(h_1)=\deg(h_2)$. \\
Indeed, let \mbox{$(\lambda_1:\mu_1)$}, $(\lambda_2:\mu_2) \not \in \sigma(h_1,h_2)$,   and set $H_i=\lambda_i h_1 - \mu_i h_2$. Then, we have  \mbox{$\deg(H_1)=\deg(H_2)=\deg(h)$}, $H_1$, $H_2$ are irreducible in $\CC[X,Y]$, and $f=u\big(v(H_1/H_2)\big)$, where $v \in \CC(T)$ is an homography.
\end{Rem}

The following proposition explains why  indecomposable functions are important in our situation.
\begin{Prop}\label{prop:ker}
 If we denote by $\ker D$ the set of rational first integrals of $D$ we have $\ker D= \CC(f/g)$ where  $f/g\in \CC(X,Y)$ is indecomposable, or $\ker D= \CC$.
\end{Prop}

A consequence of this proposition is that  any two indecomposable rational first integrals are
equal, up to an homography.\\

In the situation of a polynomial first integral, we can find a proof of this proposition in \cite[Corollary 18]{FerragutLlibre}. However
we have not found a reference for this result in the situation of  a rational first integral. Thus we give a proof  based on the following lemma, see \cite[Lemma A.1]{Singer}. This lemma means that at a nonsingular point of a vector field there is at most one algebraic solution. Another proof of Proposition \ref{prop:ker} using L\"uroth's theorem is posssible. \\

\begin{Lem}\label{Lemsing}
Let $D=A(X,Y) \partial_X + B(X,Y)\partial_Y$ be a derivation and let $f_1$, $f_2$ be two invariant algebraic curves. Suppose that $(x_0,y_0)$ is a nonsingular point of $D$, i.e. $A(x_0,y_0)\neq 0$ or $B(x_0,y_0) \neq 0$. If $f_1(x_0,y_0)=f_2(x_0,y_0)=0$ and $f_1$ is irreducible then $f_1$ divides $f_2$.
\end{Lem}

\begin{proof}[Proof of Proposition \ref{prop:ker}]
If $\ker D \neq \CC$ then there exists $f/g \in \CC(X,Y) \setminus \CC$ such that $D(f/g)=0$. We can suppose that $f/g$ is indecomposable. Indeed, if $f/g=u(h)$, with $\deg(u) \geq 2$ and $h$ indecomposable, then we have $0=D(f/g)=u'(h)D(h)$. As  $\deg(u) \geq 2$ we deduce   $u'(h)\neq0$ and $D(h)=0$. Thus $h$ is indecomposable and $h \in \ker D$. \\
Now, as $f/g$ is indecomposable,  we can suppose, thanks to Remark~\ref{remhomo} and the previous computation,  that $f$ and $g$  are irreducible in $\CC[X,Y]$ with $\deg(f)=\deg(g)$.\\
Then, we can assume that $f/g$  is a rational function with minimal degree in the set of all indecomposable rational functions $F/G \in \ker D$ such that $F$ and $G$ are irreducible in $\CC[X,Y]$ with $\deg(F)=\deg(G)$. \\

Now, let $F_1/F_2 \in \ker D$, and  set $F_1/F_2=U(H_1/H_2)$ where $H_1/H_2$ is indecomposable. As before, we can show that $D(H_1/H_2)=0$ and
 suppose, thanks to Remark \ref{remhomo}, that $H_1$ and $H_2$ are irreducible in $\CC[X,Y]$, and $\deg(H_1)= \deg(H_2)$. \\
Now we claim:
$$(\star) \textrm{   there exist } \alpha_i,\beta_i \in \CC \textrm{ such that } H_i=\alpha_i f-\beta_ig.$$ 
We  prove the claim:
As the number of singular points  of $D$ is  finite we can consider $(x_i,y_i) \in \CC^2$, where $i=1, 2$,  are nonsingular points of $D$ such that \mbox{$H_i(x_i,y_i)=0$}. Then we apply Lemma \ref{Lemsing}, to $H_i$ and $g(x_i,y_i) f(X,Y)-f(x_i,y_i)g(X,Y)$ and we deduce that $H_i$ divides $g(x_i,y_i) f(X,Y)-f(x_i,y_i)g(X,Y)$. As we have supposed  the degree of $f/g$ minimal, this proves the claim.


Now, thanks to our claim, we get $H_1/H_2=V(f/g)$ where $V \in \CC(T)$ is an homography. Thus  $F_1/F_2=U\big(V(f/g)\big)$. It follows $\ker D = \CC(f/g)$.

\end{proof}

It follows that indecomposable first integrals correspond to  first integrals with  minimal degree.




\section{Proof of Theorem \ref{Thm}}\label{sec:proof}
As mentioned before, the strategy will be the following: first construct a rational first integral with  factors some of the $f_{(\lambda,\mu),i}$ and then get a contradiction.
\begin{proof}[Proof of Theorem \ref{Thm}]
Suppose that $\rho(f,g) \geq \mathcal{B}+2$, we are going to prove that this situation is absurd.\\
We denote by $s$ the number of remarkable values in $\sigma(f,g)\setminus \gamma(f,g)$. Furthermore, we use  the following convention: if $Z$ divides $f_{(\lambda_i:\mu_i)}^h$ then we set $f_{(\lambda_i:\mu_i), n(\lambda_i,\mu_i)}^h=Z$. Then,  $f_{(\lambda_1,\mu_1),1}$, \ldots, $f_{(\lambda_1,\mu_1),n(\lambda_1:\mu_1)-1}$, \ldots, $f_{(\lambda_s,\mu_s),1}$, \ldots, $f_{(\lambda_s,\mu_s),n(\lambda_s:\mu_s)-1}$ are non-trivial invariant algebraic curves by Proposition \ref{fact-inv}.\\
As $\sum_{i=1}^s\big( n(\lambda_i:\mu_i)-1 \big)=\rho(f,g) \geq \mathcal{B}+2$, Jouanolou's theorem in the sparse case, see Theorem \ref{cheze-sparse}, implies that these curves give a rational first integral. That is to say, there exist integers $c_{(\lambda_1:\mu_1),1}$, \ldots, $c_{(\lambda_s:\mu_s),n(\lambda_s:\mu_s)-1}$ such that 
$$\prod_{j=1}^s  \, \prod_{i=1}^{n(\lambda_j:\mu_j)-1} f_{(\lambda_j:\mu_j),i}^{c_{(\lambda_j:\mu_j),i}} \in \ker D.$$
By Proposition \ref{prop:ker}, we get 
$$\prod_{j=1}^s  \, \prod_{i=1}^{n(\lambda_j:\mu_j)-1} f_{(\lambda_j:\mu_j),i}^{c_{(\lambda_j:\mu_j),i}} =u\Big(\dfrac{f}{g}\Big),$$
where $u \in \CC(T)$. Furthermore, $u(f/g)$ is a rational function of this form:
$$u\Big(\dfrac{f}{g}\Big)=\dfrac{\prod_l (\alpha_l f- \beta_l g)}{\prod_m (\gamma_m f- \delta_m g)}.g^e,$$
where $\alpha_l, \beta_l, \gamma_m, \delta_m \in \CC$ and $e \in \ZZ$. Thus we get the equality
$$(\ast) \quad \prod_{j=1}^s  \, \prod_{i=1}^{n(\lambda_j:\mu_j)-1} f_{(\lambda_j:\mu_j),i}^{c_{(\lambda_j:\mu_j),i}} =\dfrac{\prod_l (\alpha_l f- \beta_l g)}{\prod_m (\gamma_m f- \delta_m g)}.g^e.$$
As the factorization into irreducible factors in $\CC[X,Y]$ is unique, we deduce that  $(\alpha_l: \beta_l)$ and  $ (\gamma_m:\delta_m)$ belong to the spectrum $\sigma(f,g)$. Witout loss of generality we can suppose that $(\lambda_1,\mu_1)=(\alpha_1,\beta_1)$. Then we deduce that equality  $(\ast)$ is impossible since the factor $f_{(\lambda_1:\mu_1),n(\lambda_1:\mu_1)}$ does not appear in the left hand side of the equality but must appear in the right hand side. This concludes  the proof.
\end{proof}


\section*{Acknowledgment}
The author thanks Thomas Cluzeau, Jacques-Arthur Weil and Antoni Ferragut for their precious comments during the preparation of this text. Furthermore, the author has greatly appreciate the comments of the anonymous referee. These comments have allow to improve the submitted paper.


 \bibliographystyle{plain}
\bibliography{biblio-darboux-spectre}


\end{document}